\documentclass[11pt]{amsproc}
 \usepackage[margin=1in]{geometry}
\usepackage{setspace,fullpage}
\usepackage{graphicx}
\usepackage[nice]{nicefrac}
\usepackage{amssymb}
\usepackage{amsmath,amsthm,amscd,amssymb, mathrsfs}

\usepackage{latexsym}
\usepackage[colorlinks,citecolor=red,pagebackref,hypertexnames=false]{hyperref}

\numberwithin{equation}{section}

\theoremstyle{plain}
\newtheorem{theorem}{Theorem}[section]
\newtheorem{lemma}[theorem]{Lemma}
\newtheorem{corollary}[theorem]{Corollary}

\newtheorem{conjecture}[theorem]{Conjecture}

\theoremstyle{definition}

\newtheorem{example}[theorem]{Example}

\theoremstyle{remark}

\newtheorem{case[theorem]}{Case}

\title[\parbox{14cm}{\centering{ Additive energy and the Falconer distance problem in finite fields \hspace{1in}}} \quad]{Additive energy and the Falconer distance problem in finite fields }
\author{ Doowon Koh and Chun-Yen Shen }

\address{Department of Mathematics\\
Michigan State University \\
East Lansing, MI 48824,  USA}
\email{koh@math.msu.edu}

\address{Department of Applied Mathematics\\
 National Chiao Tung University\\
Hsinchu 300, Taiwan}
\email{chunyshen@gmail.com}



\begin{document}

\begin{abstract} We study the number of the vectors determined by two sets in $d$-dimensional vector spaces over finite fields.
We observe that the lower bound of cardinality for the set of vectors can be given in view of an additive energy or the decay of the Fourier transform on given sets. As an application of our observation, we find  sufficient conditions on sets where the Falconer distance conjecture for finite fields holds in two dimension. Moreover, we give an alternative proof of the theorem, due to Iosevich and Rudnev, that any Salem set satisfies the Falconer distance conjecture for finite fields. \end{abstract}
\maketitle
\tableofcontents

\section{Introduction}

Let $\mathbb F_q^d, d\geq 1,$ be $d$-dimensional vector space over the finite field $\mathbb F_q$ with $q$ elements.
Given $A, B \subset \mathbb F_q^d,$ one may ask what is the cardinality of the set $A-B,$ where the difference set $A-B$ is defined by
$$A-B=\{x-y\in \mathbb F_q^d: x\in A, y\in B\}.$$
It is clear that $|A-B|\geq \max\{ |A|, |B|\},$ here, and throughout the paper, we denote by $|E|$ the cardinality of the set $E.$
However, taking $A=B=\mathbb F_q^s, 1\leq s\leq d,$ shows that the trivial estimate for $|A-B|$ is sharp in general, because $|A-B|=|\mathbb F_q^s|=q^s.$ Moreover, if $ s=d-1$, then the size of $A-B$ is much smaller than that of $\mathbb F_q^d,$ although $|A||B|=q^{2d-2}$ is somewhat big.
Therefore, it may be interesting to find some conditions on the sets $A,B\subset \mathbb F_q^d$ such that the cardinality of $A-B$ is much bigger than
the trivial lower bound, $\max\{|A|,|B|\},$ of $|A-B|,$ or the difference set $A-B$ contains a positive proportion of all vectors in $\mathbb F_q^d,$ that is $ |A-B|\gtrsim |\mathbb F_q^d|=q^d.$ Here, we recall that for $l, m >0,$ the expression $l\gtrsim m$ or $m\lesssim l$ means that there exists a constant $ c>0$ independent of $q,$ the size of the underlying finite field $\mathbb F_q,$ such that $ cl \geq m.$
The problem to consider the size of difference sets is strongly motivated by the Falconer distance problem for finite fields, which was introduced by Iosevich and Rudnev \cite{IR07}. In this paper, we shall make an effort to find the connection between the size of the difference set $A-B$ and the cardinality of the distance set determined by $A,B\subset \mathbb F_q^d.$ As one of the main results, we shall give some examples for sets satisfying the Falconer distance conjecture for finite fields.\\

First, let us review the Falconer distance problem for the Euclidean case and the finite field case. In the Euclidean setting, the Falconer distance problem is to determine the Hausdorff dimensions of  compact sets $E, F\subset \mathbb R^d, d\geq 2,$ such that the Lebesgue measure of the distance set
$$\Delta(E,F):= \{|x-y|: x\in E, y\in F\}$$
is positive. In the case when $E=F$, Falconer \cite{Fa85} first addressed this problem and showed that if the Hausdorff dimension of the compact set $E$ is greater than $(d+1)/2$, then the Lebesgue measure of $\Delta(E,E)$ is positive. He also conjectured that  every compact set with the Hausdorff dimension $> d/2$ yields a distance set with a positive Lebesgue measure. This problem is called as the Falconer distance conjecture which has not been solved in all dimensions. The best known result for this problem is due to Wolff \cite{Wo99} in two dimension and Erdo\~{g}an \cite{Er05} in all other dimensions. They proved that if  the Hausdorff dimension of any  compact set $E \subset \mathbb R^d$ is greater than $d/2+ 1/3$, then  the Lebesgue measure of $\Delta(E,E)$ is positive. These results are a culmination of efforts going back to Falconer \cite{Fa85} in 1985 and Mattila  \cite{Ma87} a few years later. The Falconer distance problem on generalized distances was also studied in \cite{AI04}, \cite{HI05}, \cite{IL05}, \cite{IR070}, and \cite{IR09}.\\

In the Finite field setting, one can also study the Falconer distance problem. Given $A,B\subset \mathbb F_q^d, d\geq 2,$ the distance set $\Delta(A,B)$ is given by
$$ \Delta(A,B)=\{\|x-y\|\in \mathbb F_q: x\in A, y\in B\},$$
where $\|\alpha\|=\alpha_1^2+\dots +\alpha_d^2$ for $\alpha=(\alpha_1,\dots,\alpha_d)\in \mathbb F_q^d.$ It is clear that $|\Delta(A,B)|\leq q,$ because the distance set is a subset of the finite field with $q$ elements. In this setting, the Falconer distance problem is to determine the minimum value of $|A||B|$ such that  $|\Delta(A,B)|\gtrsim q.$ In the case when $A=B$, this problem was  introduced by Iosevich and Rudnev \cite{IR07} and they proved that if $A=B$ and $|A|\gtrsim q^{(d+1)/2},$ then $|\Delta(A,B)|\gtrsim q.$ It turned out in \cite{HIKR10} that if the dimension $d$ is odd, then
the theorem due to Alex and Rudnev gives the best possible result on the Falconer distance  problem for finite fields. However, if the dimension $d$ is even, then it has been believed that the aforementioned authors' result may be improved to the following conjecture.
\begin{conjecture}[Iosevich and Rudnev \cite{IR07}]\label{Alex} Let $K\subset \mathbb F_q^d$ with $d\geq 2$ even. If $|K|\geq C q^{\frac{d}{2}},$ with $C>0$ sufficiently large, then
$$ |\Delta(K,K)|\gtrsim q.$$
\end{conjecture}
This conjecture has not been solved in all dimensions. The exponent $(d+1)/2$ obtained by Iosevich and Rudnev is currently the best known result for all dimensions except two dimension. In two dimension,  this exponent was improved by $ 4/3$ (see \cite{CEHIK09} or \cite{KS10}).  We may consider the following general version of Conjecture \ref{Alex}:
\begin{conjecture}\label{KSC} Let $A,B\subset \mathbb F_q^d$ with $d\geq 2$ even. If $|A||B|\geq C q^d, $ with $C>0$ large enough, then
$$ |\Delta(A,B)|\gtrsim q.$$
\end{conjecture}

Theorem 2.1 in  \cite{Sh06} due to Shparlinski implies that if $A,B\subset \mathbb F_q^d, d\geq 2,$ and $|A||B|\gtrsim q^{d+1},$ then $|\Delta(A,B)|\gtrsim q.$ This was improved by authors \cite{KS10} who showed that if $|A||B|\gtrsim q^{8/3}$ for $A,B\subset \mathbb F_q^2,$ then $|\Delta(A,B)|\gtrsim q.$ For a variant of the Falconer distance problem for finite fields, see \cite{Vu08} and \cite{KoS10}.

\subsection{Purpose of this paper}
The goal of this paper is  to find some sets $A, B\subset \mathbb F_q^d, d\geq 2, $ for which Conjecture \ref{KSC} holds.
In general, it may not be easy to construct such examples, supporting the claim that Conjecture \ref{KSC} holds.
A well-known example is due to Iosevich and Rudnev \cite{IR07} who showed that
if $K\subset \mathbb F_q^d, d\geq 2,$ is a Salem set and $|K|\gtrsim q^{d/2},$ then $|\Delta(K,K)|\gtrsim q.$
Here, we recall that we say that $E\subset \mathbb F_q^d$ is a Salem set if for every $m\in\mathbb F_q^d \setminus \{(0,\dots,0)\}$,
$$ |\widehat{E}(m)|:= | q^{-d} \sum_{x\in E} \chi(-x\cdot m)|\lesssim \frac{\sqrt{E}}{q^d}.$$
They obtained this example by showing that the formula of $|\Delta(K,K)|$ is closely related to the decay of the Fourier transform on the set $K.$
In this paper, we take a new approach to find such examples.
First, we shall show that if $A,B\subset \mathbb F_q^d, d\geq 2$ and $|A-B|\gtrsim q^d$, then $|\Delta(A,B)|\gtrsim q.$
Second, we find some conditions on the set $A,B\subset \mathbb F_q^d$ such that $|A-B|\sim q^d.$
Thus, estimating the size of the difference set $A-B$ makes an important role.
For example, using our approach we can recover the example by Iosevich and Rudnev.
Moreover, we can find a stronger result that if one of $A,B \subset \mathbb F_q^d$ is a Salem set and $|A||B|\gtrsim q^d,$ then $A-B$ contains a positive proportion of all elements in $\mathbb F_q^d.$ In particular, our method yields that if one of $A,B\subset \mathbb F_q^2$ intersects with $\sim q$ points in an algebraic curve which does not contain any line, and $|A||B|\gtrsim q^2,$ then the sets $A,B$ satisfies Conjecture $\ref{KSC}$ in two dimension. 

\section{ Cardinality of difference sets}
In this section we introduce the formulas for the lower bound of difference sets.
Such formulas are closely related to the additive energy 
$$ \Lambda(A,B)=|\{(x,y,z,w)\in A\times A \times B\times B:  x-y+z-w=0\}|.$$
In fact, applying the Cauchy-Schwarz inequality shows that if $A, B\subset \mathbb F_q^d, d\geq 2,$  then
$$ |A|^2|B|^2=\left(\sum_{c\in  A-B} | A\cap (B+c)|\right)^2 \leq  |A-B| \sum_{c\in \mathbb F_q^d} | A\cap (B+c)|^2.$$
Observing that $\sum_{c\in \mathbb F_q^d} | A\cap (B+c)|^2=\Lambda(A,B),$ it follows that
\begin{equation}\label{basicD} 
|A-B|\geq \frac{|A|^2|B|^2}{ \Lambda(A,B)}.
\end{equation}
Since $\Lambda(A,B)\leq \min \{|A|^2|B|, |A||B|^2\},$ it is clear that 
$$ |A-B|\geq \max\{|A|, |B|\},$$
which is in fact a trivial bound of $|A-B|.$ However, if we take a subspace as $A, B $ with $A=B,$ then the trivial bound is the best bound.
In this case,  the difference set $A-B$ has much smaller cardinality than $|A||B|$. It therefore is natural to guess that if $A$ and $B$ do not contain a big subspace, then $|A-B|$ can be large. In this paper, we shall deal with this issue.\\

The lower bound of $|A-B|$ can be written in terms of the Fourier transforms on $A, B.$ To see this, using the definition of the Fourier transform and the orthogonality relation of the nontrivial additive character of $\mathbb F_q,$ observe that
$$ \Lambda(A,B) = q^{3d} \sum_{m\in \mathbb F_q^d} |\widehat{A}(m)|^2 |\widehat{B}(m)|^2,$$
Here, we recall that the Fourier transform on the set $E\subset \mathbb F_q^d$ is defined by 
$$ \widehat{E}(m)=\frac{1}{q^d}\sum_{x\in E} \chi(-x\cdot m) \quad \mbox{for}~~m\in \mathbb F_q^d,$$
where $\chi$ denotes a nontrivial additive character of $\mathbb F_q.$ Therefore, the formula (\ref{basicD}) can be replaced by 
\begin{equation}\label{basicD1}
|A-B|\geq \frac{|A|^2|B|^2}{q^{3d} \sum_{m\in \mathbb F_q^d} |\widehat{A}(m)|^2 |\widehat{B}(m)|^2 }.
\end{equation}
This formula indicates that if the Fourier decay on $A$ or $B$ is good, then several kinds of vectors are contained in the difference set $A-B.$
For example, if $A$ or $B$ takes a Salem set such as the paraboloid or the sphere, then $|A-B|$ is big and so a lot of distances can be determined by $A,B.$ 

\section{ Sets in $\mathbb F_q^2$ satisfying the Falconer distance conjecture }
In view of the sizes of difference sets, we shall find some sets $A,B\subset \mathbb F_q^2$ where the Falconer distance conjecture (Conjecture \ref{KSC}) holds.
Simple but core idea is due to the following fact.
\begin{lemma}\label{easykey} Let $E\subset \mathbb F_q^2.$ If $|E|\geq cq^2$ for some $0<c\leq 1,$ then we have
$$ |\{\|x\|\in \mathbb F_q: x\in E\}| \geq \frac{cq}{2}.$$\end{lemma}
\begin{proof} For each $a\in \mathbb F_q,$ consider a vertical line $L_a=\{(a,t)\in \mathbb F_q^2: t\in \mathbb F_q\}.$
Since $|E|\geq cq^2,$ it is clear from the pigeonhole principle that there exists a line $L_{b}$ for some $b\in \mathbb F_q$ with $|E\cap L_b|\geq cq.$ 
Thus, Lemma follows from the following observation that for the fixed $b\in\mathbb F_q,$
$$|\{b^2+t^2\in \mathbb F_q: (b, t)\in E\cap L_b\}| \geq \frac{cq}{2}.$$
\end{proof}

If $|A-B|\gtrsim |A||B|\gtrsim q^2$, then Lemma \ref{easykey} implies that $A, B\subset\mathbb F_q^2$ are the sets to satisfy the Falconer conjecture.
Thus, the main task is to fine sets $A,B$ such that $|A-B|$ is extremely large. 
The following lemma tells us some properties of  sets $A,B$ where the size of $A-B$ can be large.

\begin{lemma}\label{core1} Let $B\subset \mathbb F_q^2.$ Suppose that there exists a set $W \subset \mathbb F_q^2$ with $|W|\sim 1$ such that
\begin{equation}\label{assumption} |B \cap (B+c)|\lesssim 1 \quad \mbox{for all}~~ c\in \mathbb F_q^2\setminus W.\end{equation} Then, for any $A\subset \mathbb F_q^2,$ we have
$$ |A-B|\gtrsim \min (|A||B|, |B|^2 ).$$
\end{lemma}
\begin{proof}
From (\ref{basicD}), it suffices to show that
$$ \Lambda(A,B)= |\{(x,y,z,w)\in A\times A \times B\times B:  x-y+z-w=0\}| \lesssim  |A||B| + |A|^2.$$
It follows that
$$ \Lambda(A,B)=\sum_{x,y\in A} \left(\sum_{w,z\in B:z-w=y-x} 1\right)= \sum_{x,y\in A} |B \cap (B+ y-x)|$$
$$ =\sum_{x,y\in A: y-x\notin W}|B \cap (B+ y-x)|  + \sum_{x,y\in A: y-x\in W} |B \cap (B+ y-x)|$$
$$=\mbox{I} + \mbox{II}.$$
From the assumption (\ref{assumption}), it is clear that $|\mbox{I}|\lesssim |A|^2.$
On the other hand, the value $\mbox{II}$ can be estimated as follows.
$$ \mbox{II}= \sum_{\beta\in W} \sum_{x,y\in A: y-x=\beta} |B \cap (B+\beta)| \leq \sum_{\beta\in W} \sum_{x,y\in A: y-x=\beta} |B|.$$
Whenever we fix $x\in A$ and $\beta\in W,$  there is at most one $y\in A$ such that $y-x=\beta.$ We therefore see
$$\mbox{II} \leq |W||A||B| \sim  |A||B|.$$
Thus, we complete the proof.\end{proof}

\subsection{Examples of the Falconer conjecture sets in two dimension}
 First recall that the Bezout's theorem says that
two algebraic curves of degrees $d_1$ and $d_2$  intersect in $d_1\cdot d_2$ points and cannot meet in more than $d_1\cdot d_2$ points unless they have a component in common. As a direct application of the Bezout's theorem, it can be shown that subsets of certain algebraic curves in two dimension satisfy the condition in (\ref{assumption}). This observation yields the following theorem.
 \begin{theorem}\label{core} Let $P(x)\in \mathbb F_q[x_1,x_2]$ be an polynomial which does not have any liner factor.
Define an algebraic variety $V=\{x\in \mathbb F_q^2: P(x)=0\}.$ If $B\subset V$, then for any $A\subset \mathbb F_q^2,$ we have
$$ |A-B|\gtrsim \min( |A||B|, |B|^2).$$
\end{theorem}
\begin{proof} First recall that we always assume that the degree of the polynomial is $\sim 1.$
 Thus, if $B\subset V$, then the pigeonhole principle implies that we can choose a subvariety $V^\prime$ of $V$ and a set $B^\prime\subset V^\prime$ with $|B^\prime|\sim |B|.$ Therefore, we may assume that $V$ is a variety generated by an irreducible polynomial with degree $k\geq 2.$
Applying the Bezout's theorem shows that for any $c\in \mathbb F_q^2\setminus \{(0,0)\},$ 
$$ |V\cap (V+c)|\leq k^2\lesssim 1.$$
Therefore, the proof is complete from Lemma \ref{core1}.
\end{proof}

The following corollary follows immediately from Lemma \ref{core1} and Lemma \ref{easykey}.
\begin{corollary}\label{strange}
Let $B \subset \mathbb F_q^2$ with $|B| \gtrsim q.$ Suppose that $W\subset \mathbb F_q^2$ with $|W|\sim 1,$  and $|B \cap (B+c)| \lesssim 1$ for any $c\in \mathbb F_q^2\setminus W.$  Then, for any $A\subset \mathbb F_q^2$ with $|A|\gtrsim q,$ we have
$$|\Delta(A,B)|=|\{||x-y||\in \mathbb F_q: x \in A, y \in B\}| \gtrsim q.$$ \end{corollary}

Notice that such sets $A,B$ as in this corollary satisfy the Falconer distance conjecture.
Moreover, the difference set $A-B$ contains a positive proportion of all elements in $\mathbb F_q^2.$
As a consequence of Theorem \ref{core} and corollary \ref{strange},  more concrete examples for the Falconer distance conjecture sets can be found.
\begin{example} First,we choose a polynomial $P\in \mathbb F_q[x_1,x_2]$ which does not contain any linear factor.
Second, consider a variety $V=\{x\in \mathbb F_q: P(x)=0\}.$ If we can check that $|V|\gtrsim q $, then choose a subset $B \subset V$ with $|B|\sim q.$
Finally, choose any subset $A$ of $\mathbb F_q^2,$ whose cardinality is $\sim q.$ Then, the difference set $A-B$ contains the positive proportion of all elements in $\mathbb F_q^2$  and so $|\Delta(A,B)|\sim q.$ Since $|A||B|\sim q^2,$  the sets $A,B$ are of the Falconer distance conjecture sets.
\end{example} 

Observe that if both $A$ and $B$ contain many points in some lines $L_1, L_2$ respectively, then we can not proceed such steps as in above example.
If sets $A, B$ possess  the structures like product sets, then it seems that two sets $A,B$ determine the distance set $\Delta(A,B)$ with a small cardinality.

\section{ Salem sets and difference sets}
If the decay of the Fourier transform on $A,B \subset \mathbb F_q^d$ is known, then  the formula (\ref{basicD1}) can be very useful to measure the lower bound of $|A-B|.$  Here, we shall show that if one of $A$ and $B$ is a Salem set, then $|A-B|$ is so big that $A,B$ satisfy the Falconer distance conjecture. We need the following lemma which shows the relation between the Fourier decay of sets and the size of difference sets.

\begin{lemma}\label{decaylemma} Let $A,B\subset \mathbb F_q^d.$ Suppose that for every $m \in \mathbb F_q^d\setminus \{(0,\dots,0)\},$
\begin{equation}\label{decayassumption}|\widehat{B}(m)|\lesssim q^\beta \quad \mbox{for some} ~~\beta\in \mathbb R.\end{equation}
Then, we have
$$ |A-B|\gtrsim \min \left( q^d, \frac{|A||B|^2}{ q^{2d+2\beta}}\right).$$
\end{lemma}
\begin{proof} The proof is based on the formula (\ref{basicD1}) and discrete Fourier analysis. It follows that
$$ q^{3d} \sum_{m\in \mathbb F_q^d} |\widehat{A}(m)|^2 |\widehat{B}(m)|^2 $$
$$\leq q^{3d} |\widehat{A}(0,\dots,0)|^2 |\widehat{B}(0,\dots,0)|^2 +  q^{3d} \left(\max_{m\in \mathbb F_q^d\setminus (0,\dots,0)} |\widehat{B}(m)|^2  \right) \sum_{m\in \mathbb F_q^d} |\widehat{A}(m)|^2$$
$$ = \mbox{I} + \mbox{II}.$$
By the definition of the Fourier transform, it is clear that $\mbox{I}= q^{-d} |A|^2|B|^2.$ On the other hand, using the assumption (\ref{decayassumption}) and the Plancherel theorem, we obtain that $\mbox{II}\lesssim q^{2d+2\beta} |A|.$ Thus, we have
$$ q^{3d} \sum_{m\in \mathbb F_q^d} |\widehat{A}(m)|^2 |\widehat{B}(m)|^2 \lesssim q^{-d} |A|^2|B|^2 + q^{2d+2\beta} |A|.$$
Thus, Lemma \ref{basicD1} can be used to obtain that
$$ |A-B| \gtrsim \frac{|A|^2|B|^2}{ q^{-d} |A|^2|B|^2 + q^{2d+2\beta} |A|} \gtrsim  \min \left( q^d, \frac{|A||B|^2}{ q^{2d+2\beta}}\right),$$
which completes the proof.
\end{proof}

As mentioned in introduction, it is known that if $B\subset\mathbb F_q^d$ with $|B|\gtrsim q^{d/2}$ is a Salem set, then $|\Delta(B,B)\gtrsim q.$
Namely, the Salem set $B$ is of the Falconer distance conjecture sets. In this case, we can state a strong fact that $B-B$ contains a positive proportion of all elements in $\mathbb F_q^d.$ More precisely, we have the following theorem.
\begin{theorem} If  $B\subset \mathbb F_q^d$ is a Salem set, then for any $A\subset \mathbb F_q^d$ with $|A||B|\gtrsim q^d,$ we have
$$|A-B|\gtrsim q^d.$$
\end{theorem}

\begin{proof} Since $B\subset \mathbb F_q^d$ is a Salem set,  taking $ q^\beta =q^{-d} \sqrt{|B|}$ from Lemma \ref{decaylemma} shows that
 \begin{equation}\label{majimac} |A-B|\gtrsim \min \{ q^d, |A||B|\}.\end{equation}
 Since $|A||B|\gtrsim q^d,$ the proof is complete.
\end{proof}

The following corollary follows immediately from above theorem and Lemma \ref{easykey}.

\begin{corollary} Let $A\subset \mathbb F_q^d$ is a Salem set. Then, for any $B\subset \mathbb F_q^d$ with $|A||B|\gtrsim q^d,$ we have
$$ |\Delta(A,B)|\gtrsim q.$$
In other words, the sets $A,B$ satisfy the Falconer distance conjecture.
\end{corollary}

\end{document}